\documentclass[12pt]{article}

\usepackage{graphics}
\usepackage{graphicx}
\usepackage{epsfig}
\usepackage[all]{xy}
\usepackage{cite}

\usepackage{amssymb}
\usepackage{amsthm,amsmath}

\usepackage{latexsym, epsfig, psfrag,eepic,colordvi,bm}
\usepackage{graphics,color}
\usepackage{amsmath,amsfonts,amssymb,amscd}
\usepackage[all]{xy}
\usepackage{epstopdf}
\usepackage{mathrsfs}

\usepackage{amsthm,amsmath,amssymb}

\usepackage{graphicx}
\usepackage{amsthm,amsmath,amssymb}

\usepackage{graphicx,epsfig}

\usepackage[colorlinks=true,citecolor=black,linkcolor=black,urlcolor=blue]{hyperref}
\usepackage{arydshln}

\theoremstyle{plain}
\newtheorem{theorem}{Theorem}
\newtheorem{lemma}[theorem]{Lemma}
\newtheorem{corollary}[theorem]{Corollary}
\newtheorem{proposition}[theorem]{Proposition}

\theoremstyle{definition}

\theoremstyle{remark}

\date{}
\title{\bf Self-reciprocal polynomials connecting unsigned and signed relative derangements}

\author{Ricky X. F. Chen, Yu-Chen Ruan\\
	\small School of Mathematics, Hefei University of Technology\\[-0.8ex]
	\small Hefei, Anhui 230601, P.~R.~China\\[-0.8ex]
	\small\tt xiaofengchen@hfut.edu.cn, 1059568476@qq.com}

\begin{document}
	\maketitle
	\begin{abstract}
		In this paper, we introduce polynomials (in $t$)
		of signed relative derangements that track the number of signed elements.
		The polynomials are clearly seen to be in a sense symmetric.
		Note that relative derangements are those without any signed elements, i.e.,
		the evaluations of the polynomials at $t=0$.
		Also, the numbers of all signed relative derangements are given by the evaluations at $t=1$.
		Then the coefficients of the polynomials connect unsigned and signed relative derangements and reveal how putting elements with signs
		affects the formation of derangements.
		We first prove a recursion satisfied by these polynomials which results in a recursion satisfied by
		the coefficients. A combinatorial proof of the latter is provided next.
		We also show that the sequences of the coefficients are unimodal.
		Moreover, other results are obtained, for instance, a kind of dual of a relation between signed derangements and signed relative derangements previously proved by
		Chen and Zhang is presented.
		
		
	\end{abstract}

	\vskip 10pt
	
	\section{Introduction}
	%错排
	A derangement on a set $[n]=\{1,2,\ldots,n\}$ is a permutation $\pi=\pi_1\pi_2\cdots\pi_n$ on $[n]$ such that $\pi_i\neq i$ for all $i \in [n]$, i.e.,
	a permutation without fixed points. We use $\mathbb{D}_n$ to denote the set of derangements on $[n]$ and $D_n$ to denote the number of derangements on $[n]$.
	The study of derangements may date back to Euler who showed that the probability for a random permutation to be a derangement tends
	to $1/e$. It is also well known (e.g., Stanley~\cite[Chapter~$2$]{ref8}) that
	\begin{align}
		D_n=(n-1)(D_{n-1}+D_{n-2})\label{eq:a}.
	\end{align}
	%相对错排	
	A relative derangement $\pi=\pi_1\pi_2\cdots\pi_n$ on $[n]$ is a permutation such that $\pi_{i+1}\neq \pi_i+1$ for $1\leq i\leq n-1$.
	%$Q_n$ 
	Let	$\mathbb{Q}_n$ denote the set of relative derangements on $[n]$ and $Q_n=|\mathbb{Q}_n|$.
	With the aid of the notion of skew derangements, Chen~\cite{ref4} combinatorially showed that
	\begin{align}
		Q_n=D_n + D_{n-1}\label{eq:b}.
	\end{align}
	
	%符号排列
	A signed permutation $\pi$ on $[n]$ can be viewed as a bijection on the set $[n] \bigcup \{\overline{1},\ldots,\overline {n}\}$ such that $\pi(\overline {i})=\overline{\pi(i)}$, where $\overline{\overline{j}}=j$. Intuitively, a signed permutation on $[n]$ is just an ordinary permutation $\pi=\pi_1\pi_2\cdots\pi_n$ with some elements associated with a bar. For example, $\overline{1} 3 4 \overline{2}$ is a signed permutation on $\{1,2,3,4\}$.
	These elements with a bar are called signed elements or bar-elements.
	The set of signed permutation on $[n]$ is often denoted by $B_n$.
	%符号错排
	A signed derangement (see e.g.~\cite{ref1}) on $[n]$ is a signed permutation $\pi=\pi_1\pi_2\cdots\pi_n$ such that $\pi_i\neq i$, for all $i \in [n]$. For example, $\overline{1} 3 4 \overline{2}$ is a signed derangement in $B_4$, whereas $1 3 4 \overline{2}$ is not since it has a fixed point $1$.
	%D_n^B
	%符号相对错排
	A signed relative derangement (or sometimes called relative derangement of type $B$, see~\cite{ref5}) on $[n]$ is a signed permutation on $[n]$ such that $i$ is not followed by $i+1$, and $\overline i$ is not followed by $\overline{i+1}$.
	For example, $\overline{1} 3 \overline{2} 4$ is a signed relative derangement.
	We denote by $\mathbb{D}_n^B$ and $\mathbb{Q}_n^B$
	the sets of signed derangements and signed relative derangements on $[n]$, respectively.
	Let $D_n^B=|\mathbb{D}_n^B|$ and $Q_n^B=|\mathbb{Q}_n^B|$.
	Making use of the notion of signed skew derangements, Chen and Zhang~\cite{ref5} proved that 
	\begin{align}
		Q_n^B=D_n^B+D_{n-1}^B\label{eq:c}.
	\end{align}
	One of our results in this paper is a kind of dual of this relation, that is, we present a relation expressing $D_n^B$
	in terms of $f_n$ that counts an essential subset of sequences in $\mathbb{Q}_n^B$.
	
	Obviously, the subset of sequences with zero signed elements is $\mathbb{Q}_n$ and hence $\mathbb{Q}_n \subset \mathbb{Q}_n^B$.
	It is natural to consider the subset consisting of sequences with $m$ signed elements.
	As such, a polynomial $Q_n^B(t)$ tracking the number of signed elements is introduced.
	While many polynomials or $q$-analogues associated to derangements have been studied, for instance, the $q$-enumeration of derangements in $B_n$ by flag major index~\cite{ref1},
	the excedances of derangements~\cite{ref6,ref10}, the $q$-enumeration of derangements by major index~\cite{ref9}, and the cyclic polynomials of derangements~\cite{ref7}, our polynomials here seem to have been overlooked.
	In addition, our polynomials have a nice property, namely, they are in a sense symmetric.

	The paper is organized as follows.
	In Section~\ref{Sec2}, we introduce the symmetric polynomials $Q_n^B(t)$ and prove a recursion satisfied by them.
	Various results are then derived as a consequence.
	For instance, we obtain the expectation and variance of the number of signed elements
	contained in a random signed relative derangement.
	We also derive a partial differential equation satisfied by the generating
	function of $Q_n^B(t)$.  
	Section~\ref{Sec3} is devoted to presenting a combinatorial proof of the resulting recursion satisfied by the coefficients as well as proving a unimodality
	property.

	\section{Symmetric polynomials }\label{Sec2}

	%定义
	Let $b(\pi)$ be the number of signed elements in $\pi \in\mathbb{Q}_n^{B}$. The polynomial of signed relative derangements recording the number of signed elements is then given by
	$$
	Q_n^{B}(t)=\sum_{\pi \in \mathbb{Q}_n^B} t^{b(\pi)}= \sum_{m=0}^n q_{n,m} t^m,
	$$
	where $q_{n,m}$ denotes the number of signed relative derangements with exactly $m$ signed elements.

	It is evident that $q_{n,m}=q_{n,n-m}$ as we can obtain a signed relative derangment with $n-m$ bar-elements by
	turning a signed element into its unsigned counterpart and vice versa. Therefore, the polynomial $Q_{n}^B(t)$ is self-reciprocal.
	
	Denote by $\widetilde{\mathbb{Q}}_n^B$ the set of signed permutations on the set $[n] $
	where in each signed permutation two consecutive entries of the form $i(i+1)$ or $\overline{i} (\overline{i+1})$ for some $1\leq i <n-1$ appears exactly once.
	For example, $\overline{4} 2 3 \overline{1} \in \widetilde{\mathbb{Q}}_4^B$.
	
	For $\pi \in \mathbb{Q}_n^{B}$, we denote the resulting sequence from removing $n$ or $\overline{n}$ whichever appears in $\pi$ by $\pi^{\downarrow}$.
	The following lemma should not be hard to observe.
	%引理1
	\begin{lemma}
		For any $\pi \in \mathbb{Q}_n^{B}$, we have either
		$\pi^{\downarrow} \in \mathbb{Q}_{n-1}^{B}$ or $\pi^{\downarrow} \in \widetilde{\mathbb{Q}}_{n-1}^{B}$.
	\end{lemma}
	Accordingly, we immediately have
	\begin{align}
		Q_n^{B}(t)\label{eq:d}
		&=\sum_{\pi \in \mathbb{Q}_n^B} t^{b(\pi)}=\sum_{\pi\in \mathbb{Q}_n^{B}, \, \pi^{\downarrow}\in \widetilde{\mathbb{Q}}_{n-1}^{B} } t^{b(\pi)}+\sum_{\pi\in \mathbb{Q}_n^{B}, \, \pi^{\downarrow}\in \mathbb{Q}_{n-1}^{B}} t^{b(\pi)} .
	\end{align}
	To obtain a recursion of $Q_n^B(t)$, we next study the two sums on the right-hand side of eq.~\eqref{eq:d} in detail.
	For $\pi=\pi_1\pi_2\cdots\pi_{n-1}\in \mathbb{Q}_{n-1}^B$ and $n\geq 2$, denote by $S^{\uparrow}(\pi)$ the set of sequences in $\widetilde{\mathbb{Q}}_n^B$ that result from $\pi$ by lifting the elements larger than $\pi_i$ (for some $1\leq i\leq n-1$) by one and replacing $\pi_i$ with a length-two sequence $\pi_i(\pi_{i}+1)$, where we define the addition for bar-elements by the rule
	$
	\overline{i}+1=\overline{i+1}.
	$
		For example, for $\pi=\overline{4} 1 3 2$, $S^{\uparrow}(\pi)$ is given as follows:
		$$
		S^{\uparrow}(\pi) =\{\overline{4} \overline{5} 1 3 2, \overline{5} 1 2 4 3, \overline{5} 1 3 4 2, \overline{5} 1 4 2 3 \}.
		$$
	Moreover, if an element $x$ appears an entry in $\pi$, we write $x\in \pi$. 
	%引理2
	\begin{lemma}\label{lem:S}
		For {$n\geq 1$} and any $\pi \in \mathbb{Q}_n^{B}$, we have 
		\begin{align}
			\sum_{\pi'\in S^{\uparrow}(\pi)}t^{b(\pi')}=
			b(\pi)t^{b(\pi)+1}+(n-b(\pi))t^{b(\pi)}.
		\end{align}
	\end{lemma}
	\begin{proof}
		For any $\pi=\pi_1\pi_2\cdots\pi_n\in \mathbb{Q}_n^B$, it has $b(\pi)$ bar-elements and $n-b(\pi)$ elements without a bar. For any $\pi_i\in\pi$ with a bar, it will generate an additional bar-element after lifting the elements larger than $\pi_i$ (for some $1\leq i\leq {n}$) by one and replacing $\pi_i$ with a length-two sequence $\pi_i(\pi_{i}+1)$. In other words, it will contribute $t^{b(\pi)+1}$. However, for any $\pi_i\in\pi$ without a bar, the number of bar-elements in the sequence will not change. Therefore, it contributes $t^{b(\pi)}$.
		Summarizing the two cases gives the lemma.
	\end{proof}
	
	The lemma right below is not difficult to verify.
	%引理3
	\begin{lemma}\label{lem:S-u}
		If $\pi, \pi' \in \mathbb{Q}_{n-1}^B $ and $\pi\neq \pi'$, then $S^{\uparrow}(\pi) \bigcap S^{\uparrow}(\pi')=\varnothing$.
		Moreover,
		\begin{align}
			\widetilde{\mathbb{Q}}_{n}^B= \bigcup_{\pi \in \mathbb{Q}_{n-1}^B} S^{\uparrow}(\pi).
		\end{align}
	\end{lemma}

	%命题4
	\begin{proposition} \label{prop:w-tilde}
		For {$n\geq 2$}, we have
		\begin{align}
			\sum_{\pi\in \mathbb{Q}_n^{B}, \, \pi^{\downarrow}\in \widetilde{\mathbb{Q}}_{n-1}^{B} } t^{b(\pi)}=(1+t){\big \{}(t^2-t)Q_{n-2}^{B}{'}(t)+(n-2)Q_{n-2}^B(t){\big\}},
		\end{align}
		where $Q_n^B{'}(t)$ stands for the derivative of $Q_n^B(t)$ with respect to $t$.
	\end{proposition}
	\begin{proof}
		First, by construction, there are exactly two signed permutations $\pi, \pi'\in \mathbb{Q}_n^{B}$ such that 
		$
		\pi^{\downarrow}=\pi'^{\downarrow}\in \widetilde{\mathbb{Q}}_{n-1}^{B},
		$
		and vice versa. Specifically, if $n\in \pi$, then $\pi'$ can be obtained by replacing $n$ with $\overline{n}$ in $\pi$. Thus,
		$t^{b(\pi^{\downarrow})}=t^{b(\pi'^{\downarrow})}=t^{b(\pi)}=t^{b(\pi')-1}$ and 
		$$
		\sum_{\pi\in \mathbb{Q}_n^{B}, \, \pi^{\downarrow}\in \widetilde{\mathbb{Q}}_{n-1}^{B} } t^{b(\pi)}
		=\sum_{ \pi{'}\in \widetilde{\mathbb{Q}}_{n-1}^{B} } (1+t)t^{b(\pi{'})}.
		$$ 
		Next, we have
		\begin{align*}
			\sum_{ \pi'\in \widetilde{\mathbb{Q}}_{n-1}^{B} } t^{b(\pi')}
			&=\sum_{\pi''\in \mathbb{Q}_{n-2}^B}\sum_{\pi'\in     S^{\uparrow}(\pi'')}t^{b(\pi')}\\
			&=\sum_{\pi''\in \mathbb{Q}_{n-2}^B}{\Big\{ }b(\pi'')\cdot t+{\big [}n-2-b(\pi''){\big ]}{\Big \}}t^{b(\pi'')}\\
			&=\sum_{\pi''\in \mathbb{Q}_{n-2}^B}\Big\{(t-1)b(\pi'')t^{b(\pi'')}+(n-2)t^{b(\pi'')}\Big\}\\
			&=(t^2-t)Q_{n-2}^{B}{'}(t)+(n-2)Q_{n-2}^B(t),
		\end{align*}
		where the first two equalities follow from Lemma~\ref{lem:S} and Lemma~\ref{lem:S-u}, respectively,
		and then the proof follows.
	\end{proof}
	
	%命题5
	\begin{proposition}\label{prop:no-tilde}
		For {$n\geq 2$}, we have
		\begin{align}\label{8}
			\sum_{\pi \in \mathbb{Q}_n^{B}, \, \pi^{\downarrow}\in \mathbb{Q}_{n-1}^{B}} t^{b(\pi)} = (nt+n-1)Q_{n-1}^B(t)+ (1-t)\sum_{\pi'\in \mathbb{Q}_{n-1}^B, \, {\overline{n-1} \in \pi'}} t^{b(\pi')}.
		\end{align}
	\end{proposition}
	\begin{proof}
		A sequence $\pi \in \mathbb{Q}_n^B$ where $n$ appears can be clearly obtained by inserting $n$ into a sequence $\pi^{\downarrow}\in \mathbb{Q}_{n-1}^B$.
		We distinguish two cases:
		\begin{itemize}
			\item 	 	 if $n-1$ appears in $\pi^{\downarrow} \in \mathbb{Q}_{n-1}^B$, there are $n-1$ positions where $n$ can be inserted.
			\item if $\overline{n-1}$ appears in $\pi^{\downarrow} \in \mathbb{Q}_{n-1}^B$, there are $n$ positions where $n$ can be inserted.
		\end{itemize} 	
		Note that in both cases, we have $b(\pi)=b(\pi^{\downarrow})$. Thus,
		\begin{align*}
			\sum\limits_{\pi\in \mathbb{Q}_n^B, \, {n\in\pi}, \, \pi^{\downarrow}\in \mathbb{Q}_{n-1}^{B}}t^{b(\pi)}
			&=\sum\limits_{\pi'\in \mathbb{Q}_{n-1}^B, \,  {n-1}\in\pi'}(n-1)\cdot t^{b(\pi')}+\sum\limits_{\pi'\in \mathbb{Q}_{n-1}^B, \,  {\overline{n-1}}\in\pi'}n\cdot t^{b(\pi')}\\
			&=(n-1)Q_{n-1}^B(t)+\sum\limits_{\pi'\in \mathbb{Q}_{n-1}^B, \, {\overline{n-1}}\in\pi'}t^{b(\pi')} \, .
		\end{align*}
		Similarly, the situation of inserting $\overline n$ can be calculated.
		We also distinguish two cases:
		
		\begin{itemize}
			\item	if $n-1$ appears in $\pi^{\downarrow}\in \mathbb{Q}_{n-1}^B$, there are $n$ positions where $\overline n$ can be inserted.
			\item if $\overline{n-1}$ appears in $\pi^{\downarrow}\in \mathbb{Q}_{n-1}^B$, there are $n-1$ positions where $\overline n$ can be inserted.
		\end{itemize}	
		The difference is that in this case, we have $b(\pi)=b(\pi^{\downarrow})+1$. Thus,
		\begin{align*}
			\sum\limits_{\pi\in \mathbb{Q}_n^B, \, {\overline{n}\in\pi}, \,  \pi^{\downarrow}\in \mathbb{Q}_{n-1}^{B}}t^{b(\pi)}
			&=\sum\limits_{\pi'\in \mathbb{Q}_{n-1}^B ,{n-1}\in\pi'}nt\cdot t^{b(\pi')}+\sum\limits_{\pi'\in \mathbb{Q}_{n-1}^B, {\overline{n-1}}\in\pi'}(n-1)t\cdot t^{b(\pi')}\\
			&=ntQ_{n-1}^B(t)-t\sum\limits_{\pi'\in \mathbb{Q}_{n-1}^B , {\overline{n-1}}\in\pi'}t^{b(\pi')}.
		\end{align*}
		Combining the above two cases, we obtain the proposition.
	\end{proof}
	
	%命题6
	\begin{proposition}\label{prop:D}
		For $n\geq 3$, we have 
		\begin{align}
			\sum_{\pi'\in \mathbb{Q}_{n-1}^B, \, {\overline{n-1} \in \pi'}} t^{b(\pi')}
			=&(n-1)tQ_{n-2}^B(t)+t\big{\{}(t^2-t)Q_{n-3}^{B}{'}(t)+(n-3)Q_{n-3}^B(t)\big{\}}\nonumber\\
			&-t\sum_{\pi'' \in \mathbb{Q}_{n-2}^B, \, \overline{n-2}\in \pi''} t^{b(\pi'')} \, .
		\end{align}
	\end{proposition}
	\begin{proof}
		Analogously, we first have
		\begin{align*}
			\sum\limits_{\pi'\in \mathbb{Q}_{n-1}^B,\, {\overline{n-1}}\in\pi'}t^{b(\pi')}
			=&\sum_{\pi'\in \mathbb{Q}_{n-1}^B, \,  {\overline{n-1}}\in\pi', \, \pi'^{\downarrow}\in \mathbb{Q}_{n-2}^B} t^{b(\pi'^{\downarrow})}
			+\sum\limits_{\pi'\in \mathbb{Q}_{n-1}^B, \,  {\overline{n-1}}\in\pi', \, \pi'^{\downarrow} \in \widetilde{\mathbb{Q}}_{n-2}^{B}}t^{b(\pi'^{\downarrow})}.
		\end{align*}
		The first sum of the right-hand side has been obtained in Proposition~\ref{prop:no-tilde} and equals
		\begin{align*}
			&(n-1)tQ_{n-2}^B(t)-t\sum_{\pi'' \in \mathbb{Q}_{n-2}^B, \,  \overline{n-2}\in{\pi''}} t^{b(\pi'')}.
		\end{align*}
		Following the proof of Proposition~\ref{prop:w-tilde}, the second sum of the right-hand side equals
			\begin{align*}
				&\sum\limits_{\pi'' \in \widetilde{\mathbb{Q}}_{n-2}^{B}}t\cdot t^{b(\pi'')}
				=t\sum_{\pi'''\in \mathbb{Q}_{n-3}^B}\sum_{\pi''\in     S^{\uparrow}(\pi''')}t^{b(\pi'')}\\
				&=t\sum_{\pi'''\in \mathbb{Q}_{n-3}^B}{\Big \{}b(\pi''')\cdot t+\big{[}n-3-b(\pi''')\big{]}{\Big \}}t^{b(\pi''')}\\
				&=t\sum_{\pi'''\in \mathbb{Q}_{n-3}^B}\Big\{(t-1)b(\pi''')t^{b(\pi''')}+(n-3)t^{b(\pi''')}\Big\}\\
				&=t{\big [}(t^2-t)Q_{n-3}^{B}{'}(t)+(n-3)Q_{n-3}^B(t){\big ]}.
			\end{align*}
		The rest is clear and the proof follows.
	\end{proof}	
	
	Based on Proposition~\ref{prop:w-tilde}--\ref{prop:D}, we conclude
	
	%定理7 Q_n^B(t)
	\begin{theorem}\label{main}
		For $n\geq3$, the following holds
		\begin{align}\label{eq:qnt-recur}
			Q_n^B(t)
			&=(n-1)(t+1)Q_{n-1}^B(t)+\big\{(3n-5)t+(n-2)\big\}Q_{n-2}^{B}(t) \nonumber\\
			&\quad +(t^3-t)Q_{n-2}^{B}{'}(t) 
			+(2n-6)tQ_{n-3}^{B}(t)+2t^2(t-1)Q_{n-3}^{B}{'}(t),
		\end{align}
		and $Q_0^B(t)=0$, $Q_1^B(t)=1+t$, $Q_2^B(t)=t^2+4t+1$.
	\end{theorem}	
	\begin{proof}
		According to Proposition~~\ref{prop:w-tilde}--\ref{prop:D}, we first obtain
		\begin{align}
			&Q_n^{B}(t)
			=\sum_{\pi \in \mathbb{Q}_n^B} t^{b(\pi)}
			=\sum_{\pi\in \mathbb{Q}_n^{B}, \, \pi^{\downarrow}\in \widetilde{\mathbb{Q}}_{n-1}^{B} } t^{b(\pi)}+\sum_{{\pi\in \mathbb{Q}_n^{B}}, \,  {\pi^{\downarrow}\in \mathbb{Q}_{n-1}^{B}}} t^{b(\pi)} \nonumber\\
			=&(1+t)\big{\{}(t^2-t)Q_{n-2}^{B}{'}(t)+(n-2)Q_{n-2}^B(t)\big{\}}+(nt+n-1)Q_{n-1}^B(t)+(1-t)\sum_{\pi'\in \mathbb{Q}_{n-1}^B \atop  {\overline{n-1} \in \pi'}} t^{b(\pi')}\nonumber\\
			=&(1+t)\big{\{}(t^2-t)Q_{n-2}^{B}{'}(t)+(n-2)Q_{n-2}^B(t)\big{\}}+(nt+n-1)Q_{n-1}^B(t)\nonumber\\  &+(1-t)\Big{\{}(n-1)tQ_{n-2}^B(t)+t\big{\{}(t^2-t)Q_{n-3}^{B}{'}(t)\nonumber+(n-3)Q_{n-3}^B(t)\big{\}}-t\sum_{\pi''\in \mathbb{Q}_{n-2}^B\atop \overline{n-2}\in \pi''} t^{b(\pi'')} \, \Big{\}}\nonumber .
		\end{align}
		Iterating using Proposition~\ref{prop:D} and using the fact that $\sum\limits_{\pi\in \mathbb{Q}_{1}^B ,\,  {\overline{1} \in \pi}} t^{b(\pi)}=t$, we have
		\begin{align}
			Q_n^{B}(t)\label{eq:e}	=  &(nt+n-1)Q_{n-1}^B(t)+(1-t)\Big{\{}\sum\limits_{k=1}^{n-2}(-1)^{k+1}(n-k)t^kQ_{n-k-1}^B(t)\Big{\}}\nonumber\\
			&+(-1)^n(1-t)t^{n-1}+(1+t)\big{\{}(t^2-t)Q_{n-2}^{B}{'}(t)+(n-2)Q_{n-2}^{B}(t)\big{\}}\nonumber\\
			&+(1-t)\Big\{\sum\limits_{k=1}^{n-2}(-1)^{k+1}t^k\big{\{}(t^2-t)Q_{n-k-2}^{B}{'}(t)+(n-k-2)Q_{n-k-2}^{B}(t)\big{\}}\Big\}\nonumber\\
			= &(nt+n-1)Q_{n-1}^B(t)+(2n-4)Q_{n-2}^{B}(t)+(-1)^n(1-t)t^{n-1}\nonumber\\
			&+\sum_{k=1}^{n-2}(-1)^{k}t^{k-1}\Big{\{}(n-k-1)+(2k+1-2n)t+(n-k)t^2\Big{\}}Q_{n-k-1}^{B}(t)\nonumber\\
			&+(t^3-t)Q_{n-2}^{B}{'}(t)+\sum_{k=1}^{n-2}(-1)^{k+1}t^{k+1}(2t-1-t^2)Q_{n-k-2}^{B}{'}(t).
		\end{align}
		The above second expression follows from collecting terms from the first one. Consequently, we have
		\begin{align*}
			Q_{n-1}^{B}(t)
			=&\big{[}(n-1)t+n-2\big{]}Q_{n-2}^{B}(t)+(2n-6)Q_{n-3}^{B}(t)+(-1)^{n-1}(1-t)t^{n-2}\\
			&+\sum_{k=1}^{n-3}(-1)^{k}t^{k-1}\Big{\{}(n-k-2)+(2k+3-2n)t+(n-k-1)t^2\Big{\}}Q_{n-k-2}^{B}(t)\\
			&+(t^3-t)Q_{n-3}^{B}{'}(t) +\sum_{k=1}^{n-3}(-1)^{k+1}t^{k+1}(2t-1-t^2)Q_{n-k-3}^{B}{'}(t).
		\end{align*}
		Then, it is observed that the two sums in the last expression of eq.~\eqref{eq:e} equals
		\begin{align*}
			(-t)\Big{\{}
			&Q_{n-1}^{B}(t)-\big{[}(n-1)t+n-2\big{]}Q_{n-2}^{B}(t)\\
			&-(2n-6)Q_{n-3}^{B}(t)-(-1)^{n-1}(1-t)t^{n-2}-(t^3-t)Q_{n-3}^{B}{'}(t)\Big{\}}.
		\end{align*}
		Plugging it into eq.~\eqref{eq:e} and simplifying completes the proof.
	\end{proof}
	
	Based on the obtained recursion eq.~\eqref{eq:qnt-recur}, the first few polynomials of $Q_n^B(t)$ are computed and listed below:	
	\begin{align*}
		Q_1^B(t)=&t+1\\
		Q_2^B(t)=&t^2+4t+1\\
		Q_3^B(t)=&3t^3+14t^2+14t+3\\
		Q_4^B(t)=&11t^4+64t^3+112t^2+64t+11\\
		Q_5^B(t)=&53t^5+362t^4+866t^3+866t^2+362t+53\\
		Q_6^B(t)=&309t^6+2428t^5+7252t^4+10300t^3+7252t^2+2428t+309\\
		Q_7^B(t)=&2119t^7+18806t^6+66854t^5+121838t^4+121838t^3+66854t^2+18806t+2119\\
		Q_8^B(t)=&16687t^8+165016t^7+677656t^6+1497880t^5+1937368t^4+1497880t^3\\&+677656t^2+165016t
		+16687\\
		Q_9^B(t)=&148329t^9+1616786t^8+7513658t^7+19444106t^6+30752450t^5+30752450t^4\\&+19444106t^3
		+7513658t^2+1616786t+148329
	\end{align*}

	%推论8（生成函数）
	\begin{corollary}\label{cor:dif-eq}
		Let $F(x,t)=\sum_{n\geq 1}Q_n^B(t)x^n$ be the generating function of $Q_n^{B}(t)$. Then,
		$F(0,t)=0$ and $F(x,t)$ satisfies the following differential equation: 
		
		\begin{small}
			\begin{align}
				&\frac{\partial F}{\partial t}(x,t)+\frac{t+1+3tx+x+2tx^2}{t(t^2-1)+2t^2(t-1)x}\frac{\partial F}{\partial x}(x,t) \nonumber \\
				=& \frac{-1-t-2tx}{t(t^2-1)x+2t^2(t-1)x^2} -\frac{tx^2-1}{t(t^2-1)x^2+2t^2(t-1)x^3}F(x,t).
			\end{align}
		\end{small}
	\end{corollary}
	
	The proof of Corollary~\ref{cor:dif-eq} is provided in the appendix.
	Unfortunately, we are unable to solve the differential equation to get explicit formulas
	for $F(x,t)$ and $Q_n^B(t)$.

	%推论9（期望方差)
	\begin{corollary}
		Let $\pi \in \mathbb{Q}_n^B$ be chosen uniformly at random. Then,
		the expectation and variance of the number of signed elements $b(\pi)$ are respectively
		$$
		\mathrm{E}[b(\pi)]=\frac{n}{2}, \qquad \mathrm{Var}[b(\pi)]=F_n+\frac{2n-n^2}{4},
		$$
		where $F_n$ satisfies 
		\begin{align*}
			F_n
			&=\big{[}(n-1)^2+(2n-2)F_{n-1}\big{]}\frac{{Q_{n-1}^B}}{Q_n^B}+\big{[}(3n-2)(n-2)+(4n-3)F_{n-2}\big{]}\frac{{Q_{n-2}^B}}{Q_n^B}\\
			&\qquad +\big{[}(2n-2)(n-3)+(2n-2)F_{n-3}\big{]}\frac{{Q_{n-3}^B}}{Q_n^B}.
		\end{align*}
	\end{corollary}
	\begin{proof}
		Recall that $q_{n,m}=q_{n,n-m}$, and it is easy to see
		\begin{align*}
			Q_n^B(1)&=\sum\limits_{m=0}^{n}q_{n,m},\\
			{Q_n^B}'(t)&=\sum\limits_{m=0}^{n}mq_{n,m}t^{m-1}, \quad {Q_n^B}'(1)=\sum\limits_{m=0}^{n}mq_{n,m},\\
			{Q_n^B}''(t)& =\sum\limits_{m=0}^{n}m(m-1)q_{n,m}t^{m-2}, \quad {Q_n^B}''(1)=\sum\limits_{m=0}^{n}m(m-1)q_{n,m}.
		\end{align*}
		Consequently, we have
		\begin{align*}
			\mathrm{E}[b(\pi)]=\frac{\sum\limits_{m=0}^{n}mq_{n,m}}{\sum\limits_{m=0}^{n}q_{n,m}}{=\frac{{Q_n^B}'(1)}{{Q_n^B}(1)}=\frac{\sum\limits_{m=0}^{n}(m+n-m)q_{n,m}/2}{\sum\limits_{m=0}^{n}q_{n,m}}}=\frac{n}{2}.
		\end{align*}
		As for the variance, we compute	
		\begin{align*}
			\mathrm{Var}[b(\pi)]
			&=\frac{\sum\limits_{m=0}^{n}(m-\mathrm{E}[b(\pi)])^2q_{n,m}}{{\sum\limits_{m=0}^{n}q_{n,m}}}\\
			&=\frac{\sum\limits_{m=0}^{n}m^2q_{n,m}+\sum\limits_{m=0}^{n}{\mathrm{E}[b(\pi)]}^2q_{n,m}-2\sum\limits_{m=0}^{n}m\mathrm{E}[b(\pi)]q_{n,m}}{Q_n^B{(1)}}\\
			&=\frac{\sum\limits_{m=0}^{n}\big{[}m(m-1)+m\big{]}q_{n,m}+\sum\limits_{m=0}^{n}{\mathrm{E}[b(\pi)]}^2q_{n,m}-2\sum\limits_{m=0}^{n}m\mathrm{E}[b(\pi)]q_{n,m}}{Q_n^B{(1)}}\\
			&=\frac{{Q_n^B}''(1)+{Q_n^B}'(1)+\mathrm{E}[b(\pi)]^2{Q_n^B}(1)-2\mathrm{E}[b(\pi)]{Q_n^B}'(1)}{{Q_n^B{(1)}}}\\
			&=\frac{{Q_n^B}''(1)}{{Q_n^B{(1)}}}+\frac{2n-n^2}{4}.
		\end{align*}	
		From Theorem~\ref{main}, we next get 
		\begin{align*}
			{Q_n^B}''(1)
			=&(2n-2){Q_{n-1}^B}'(1)+(2n-2){Q_{n-1}^B}''(1)+(6n-4){Q_{n-2}^B}'(1)+(4n-3){Q_{n-2}^B}''(1)\\
			&+(4n-4){Q_{n-3}^B}'(1)+(2n-2){Q_{n-3}^B}''(1).
		\end{align*}
		By dividing both sides by $Q_n^B{(1)}$, the following recurrsion of $F_n=\frac{{Q_n^B}''(1)}{{Q_n^B}(1)}$ can be obtained:
		\begin{align*}
			F_n
			&=\big{[}(n-1)^2+(2n-2)F_{n-1}\big{]}\frac{{Q_{n-1}^B}(1)}{Q_n^B(1)}+\big{[}(3n-2)(n-2)+(4n-3)F_{n-2}\big{]}\frac{{Q_{n-2}^B}(1)}{Q_n^B(1)}\\
			&+\big{[}(2n-2)(n-3)+(2n-2)F_{n-3}\big{]}\frac{{Q_{n-3}^B}(1)}{Q_n^B(1)}.
		\end{align*}
		This completes the proof.
	\end{proof}
	
	The following corollary follows from Theorem~\ref{main} as well.
	
	%推论10 Q_n,Q_n^B,q_{n,m}
	\begin{corollary}\label{coro}
		For $n\geq 3 $ and $m\geq 0$, we have
		\begin{align}
			Q_n\label{eq:f}
			=&(n-1)Q_{n-1}+(n-2)Q_{n-2},\\
			Q_n^B\label{eq:g} =&(2n-1)Q_{n-1}^B+(2n-4)Q_{n-2}^B, \\
			q_{n,m}\label{eq:h}
			=&(n-1)q_{n-1,m-1}+(n-1)q_{n-1,m}+(m-2)q_{n-2,m-2}+(3n-5)q_{n-2,m-1}\nonumber\\
			&+(n-m-2)q_{n-2,m}+(2m-4)q_{n-3,m-2}+(2n-2m-4)q_{n-3,m-1},
		\end{align}
		where we make the convention that $q_{n,m}=0$ if $m<0$.
	\end{corollary}
	\begin{proof}
		Eq.~\eqref{eq:f} and~\eqref{eq:g} follow from eq.~\eqref{eq:qnt-recur} by setting $t=0$ and $t=1$, respectively.
		Eq.~\eqref{eq:h} is obtained by equating the coefficients of $t^m$ on both sides of eq.~\eqref{eq:qnt-recur}
	\end{proof}
	It is easy to see that the case $m=0$ of eq.~\eqref{eq:h} agrees with eq.~\eqref{eq:f}.
	Of course, eq.~\eqref{eq:f} and~\eqref{eq:g} can be also obtained by making use of the recursions satisfied by
	$D_n$, $D_n^B$, eq.~\eqref{eq:b} and eq.~\eqref{eq:c}. We leave the computation to the interested reader.	
	In the next section, we will present a direct combinatorial proof of the recursion of $q_{n,m}$.
	
	\section{Recursion and unimodality of $q_{n,m}$}\label{Sec3}
	The goal of this section is to first prove the recursion of $q_{n,m}$ combinatorially,
	and then prove the sequence of $q_{n,m}$ is unimodal.

	Before we proceed, we present a connection to the work of the first author~\cite{ref3}
	using a slight variation of signed relative derangements.
	Recall the definitions there:
	Let 
	$$
	\Gamma_n=\{(0,-1),(-1,0),(1,-2),(-2,1),\ldots,(n,-n-1),(-n-1,n)\}
	$$
	be a set of ordered pairs. For an ordered pair $T=(a,b)$, the element $a$ is called the left entry of $T$ and denoted by $T^l=a$, while $b$ the right entry of $T$ and denoted by $T^r=b$. A signed relative derangement (SRD) on $\Gamma_n$ is a sequence $\pi=T_0T_1\cdots T_n$ such that $T_i\in \Gamma_n$, each ordered pair appears at most once in $\pi$, $(a,b)\in\Gamma_n$ and $(b,a)\in\Gamma_n$ cannot be both contained in $\pi$, and for $0\leq i\leq n-1$, $T_i^r\neq-T_{i+1}^l$.
	This particular form for SRDs was chosen for a reason, as SRDs were also treated as fixed point involutions in~\cite{ref3}.
	As such, the first author could provide an upper bound for the number of signed permutations whose reversal distances are maximum possible.
	
	%fn
	An SRD of type 1 on $\Gamma_n$ is an SRD $\pi=T_0T_1T_2\cdots T_n$ such that $T_0=(0,-1)$ and $T_n\neq (n,-n-1)$. An SRD of type 2 on $\Gamma_n$ is an SRD $\pi=T_0T_1T_2\cdots T_n$ such that $T_0=(0,-1)$ and $T_n= (n,-n-1)$.
	Let $f_n$ and $\hat f_n$ denote the number of SRDs of type 1 and type 2 on $\Gamma_n$, respectively.
	%已知fn相关结果
	Clearly, $\hat f_n=f_{n-1}$. One of the main results in Chen~\cite{ref3} is the four-term recursion below
	\begin{equation}\label{chen-19}
		f_n=(2n-2)f_{n-1}+(4n-3)f_{n-2}+(2n-2)f_{n-3}, \quad (n\geq 4)
	\end{equation}
	where $f_1=1,f_2=4,f_3=25$. 
	
	Following~\cite{ref3}, we have known that there is a natural bijection for transfroming SRDs on $\Gamma_n$ to the signed relative derangements in the classical definition. That is, just view $(i,-i-1)$ as $i$ and $(-i-1,i)$ as $\overline{i}$. But it is worth noting that the condition now becomes that $i$ is not followed by $i+1$ and $\overline{i+1}$ is not followed by $\overline{i}$.
	Sometimes it is more convenient to use this definition. For instance,
	let $\pi^{[r]}$ denote the sequence obtained from $\pi$ by reading $\pi$ reversely (i.e., right to left) and
	changing $i$ to $\overline{i}$ and vice versa.
	Then, if $\pi$ is an SRD, then $\pi^{[r]}$ is also an SRD. For example, for an SRD $\pi=\overline{2} \overline{3} 1 0$, $\pi^{[r]}=\overline{0} \overline{1} 3 2$ is an SRD too.
	We refer to $\pi^{[r]}$ as the conjugate-reverse of $\pi$.
	This is not true in the classical definition. For example, for a signed relative derangement $\pi=\overline{3} \overline{2} 1 0$, $\pi^{[r]}=\overline{0} \overline{1} 2 3$ is not a signed relative derangement anymore in the classical definition.
	In the following, we will use the new version of SRDs if not explicitly stated otherwise.
	
	%引理11 Q_n^B与f_n关系
	\begin{lemma}\label{lem:Q-f}
		For $n\geq3$,
		\begin{align}\label{eq:i}
			Q_n^B=(f_n+f_{n-1})+(f_{n-1}+f_{n-2}).
		\end{align}
		
	\end{lemma}
	\begin{proof}
		
		The elements $\pi_1\pi_2 \cdots \pi_n$ in $\mathbb{Q}_n^B$ consist of two classes: $\pi_1=1$ and $\pi_1 \neq 1$.
		The latter is equivalent to SRDs of type $1$ and type $2$ and counted by $f_n+\hat{f_n}=f_n+f_{n-1}$ as discussed above.
		As for those starting with $1$, the subsequence $\pi_2\cdots \pi_n$ must not start with
		$2$. It is then not hard to see that this class is counted by $f_{n-1}+f_{n-2}$,
		completing the proof. 		
	\end{proof}
	
	In view of Lemma~\ref{lem:Q-f}, the {`}core{'} of $\mathbb{Q}_n^B$ is really the subset of sequences not starting with 1.
	Also, recall that $Q_n^B=D_n^B+D_{n-1}^B$ obtained by Chen and Zhang~\cite{ref5}.
	Accordingly, it suggests the following relation which can be viewed as a dual of this relation.
	
	%命题12
	\begin{proposition}[Dual of eq.~\eqref{eq:c}]\label{dual} 
		For $n\geq2$, we have
		\begin{align}
			D_n^B=f_n+f_{n-1}.
		\end{align}
	\end{proposition}
	\begin{proof}
		First, we take the opportunity to present a direct combinatorial proof of a recursion of $D_n^B$ which is an analogue of eq.~\eqref{eq:a}. Consider signed derangements of length $n$ in $\mathbb{D}_n^B$. We
		distinguish the following cases.\\
		\textbf{\textit{case~$1$:}} If $1$ appears, it can be placed at any other $n-1$ positions except the first position. Suppose $1$ is placed at the $k$-th position for a fixed $1<k \leq n$, then we consider the elements $k$ and $\overline{k}$.
		\begin{itemize}
			\item If $k$ is placed at the first position, the remaining $n-2$ entries (other than the first and the $k$-th entries) could essentially form any signed derangement
			of length $n-2$. Then, we have $ D_{n-2}^B$ signed derangements in this case. 
			\item If $k$ is not placed at the first position (note that $\overline{k}$ could still be placed at the first position), viewing $k$ as $1$ (and $\overline{k}$ as $\overline{1}$), the remaining $n-1$ entries other than the $k$-th entry essentially
			form a signed derangement of length $n-1$.
			Hence, there are $D_{n-1}^B$ signed derangements in this case.
		\end{itemize}
		Since there are $n-1$ options for $k$, we have $(n-1)(D_{n-2}^B+D_{n-1}^B)$ signed derangements where $1$ appears.\\
		\textbf{\textit{case~$2$:}} Consider the case $\overline{1}$ appears. 
		\begin{itemize}
			\item Clearly, there are $D_{n-1}^B$ signed derangements where $\overline{1}$ is placed at the first position.
			\item If $\overline{1}$ is not placed at the first position, in analogy with case~$1$, we have $(n-1)(D_{n-2}^B+D_{n-1}^B)$ such signed derangements.
		\end{itemize}
		
		Summarizing the above discussion, we have
		\begin{align}\label{eq:der-B}
			D_n^B=(2n-1)D_{n-1}^B+(2n-2)D_{n-2}^B.
		\end{align}	
		Next, let $F_n=f_n+f_{n-1}$. Applying the four-term recurrence eq.~\eqref{chen-19}, we have
		\begin{align*}
			F_n
			&=(2n-1)f_{n-1}+(4n-3)f_{n-2}+(2n-2)f_{n-3}\\
			&=(2n-1)F_{n-1}+(2n-2)F_{n-2}.
		\end{align*}
		That is, $D_n^B$ and $F_n$ satisfy the same recursion.
		Meanwhile, we have $D_2^B=F_2=5$, $D_3^B=F_3= 29$. Therefore, $D_n^B$ and $f_n+f_{n-1}$ also have the same initial values.
		Thus, it is proved that $D_n^B=f_n+f_{n-1}$.
	\end{proof}
	
	We remark that eq.~\eqref{eq:der-B} can be found in~\cite{ref2}, but with a different proof.
	Combining eq.~\eqref{eq:i} and eq.~\eqref{chen-19}, we immediately have an alternative proof of eq.~\eqref{eq:g}.
	
	Now we are in a position to prove the recursion eq.~\eqref{eq:h}.
	Let $\overline{q}_{n,m}$ denote the number of $\pi=\pi_1 \pi_2 \cdots \pi_n\in\mathbb{Q}_n^B$ with $m$ bar-elements and $\pi_1\neq 1$.
	Equivalently, $\overline{q}_{n,m}$ counts SRDs of type $1$ and $2$ on $\Gamma_n$ that
	have $m$ bar-elements. We first have the following
	relation which is an analogue of eq.~\eqref{eq:i}.
	
	%引理13 减序列长度
	\begin{lemma}\label{reduce}
		For $n>0$ and $0 \leq m \leq n$,
		\begin{align}
			q_{n,m}=\overline{q}_{n,m}+\overline{q}_{n-1,m}.
		\end{align}
	\end{lemma}
	\begin{proof}
		For any $\pi\in \mathbb{Q}_n^B$ with $m$ bar-elements, $\pi$ is either in the form $\pi_1\pi_2\cdots\pi_n$ where $\pi_1\neq1$ or $1\pi_2\cdots\pi_n$. The number of the former is just $\overline{q}_{n,m}$. And the number of the latter is equal to the number of
		$\pi_2\cdots\pi_n$ where $\pi_2\neq2$, namely $\overline{q}_{n-1,m}$, whence the lemma.
	\end{proof}

	In the light of Lemma~\ref{reduce}, in order for studying $q_{n,m}$ it suffices to study $\overline{q}_{n,m}$.
	To that end, we generalize the idea for proving eq.~\eqref{chen-19} in~\cite{ref3} and
	obtain

	%定理14 f_{n,k}递推
	\begin{theorem}\label{fnk}
		For $n\geq {3}$ and $0 \leq m \leq n$, we have
		\begin{align}
			\overline{q}_{n,m}
			=&(n-1)\overline{q}_{n-1,m}+(n-m-1)\overline{q}_{n-2,m}+(m-1)\overline{q}_{n-2,m-1}+n\overline{q}_{n-1,n-m}\nonumber\\
			&+(m-1)\overline{q}_{n-2,n-m}+(n-m-1)\overline{q}_{n-2,n-m-1},
		\end{align}
		where $\overline{q}_{x,y}=0$ if $y<0$ or {$y>x$, and the initial values here are determined by the coefficients of
			$Q_1^B(t)$ and $Q_2^B(t)$ in Theorem~\ref{main}}.
	\end{theorem}
	\begin{proof}
		Note that SRDs of type $1$ and $2$ on $\Gamma_n$ with $m$ bar-elements (counted by $\overline{q}_{n,m}$) are either in the form
		$0 A_1 1 A_2$ or $0 A_1 \overline{1} A_2$. We will count SRDs in each case separately.\\
		\\
		\textbf{\textit{case~$1$: }}$0 A_1 1 A_2$. 
		
		(i) Suppose $A_2=\emptyset$.
		In this case, $A_1$ could essentially (i.e., by appropriate relabelling) be any SRD of length $n-1$ with $m$ bar-elements.
		It is easy to see there are $\overline{q}_{n-1,m} + \overline{q}_{n-2,m}$ such SRDs.
		
		(ii) Suppose $A_2\neq\emptyset$. Consider the induced sequence $1 A_2 A_1$.
		
		If there exists no $a\in[n]$ such that $A_2$ ends with $\overline{a}$ while $A_1$ starts with $\overline{a-1}$ or $A_2$ ends with $a-1$ while $A_1$ starts with $a$, then the sequence $1 A_2 A_1$ could be equivalently any SRD of type $1$ or $2$ of length $n-1$ and with $m$ bar-elements. The latter is counted by $\overline{q}_{n-1,m}$. Moreover, there are $n-2$ ways to transform each such a sequence into sequences of the form $A_1 1 A_2$. Hence, there are $(n-2)\overline{q}_{n-1,m}$ SRDs lying in this situation.
		
		If otherwise, such an $a$ exists, then by construction $a\in[n] \setminus [2]$.
		That is, it is impossible to have patterns $\overline{1} \overline{0}, \, 01, \, \overline{2}\overline{1}, \, 12$ in
		$A_2A_1$ since $1$ has already been used. We claim that for a fixed $a\in[n] \setminus [2]$,
		\begin{itemize}
			\item  the sequences of the form $1 A_2' \overline{a} \overline{a-1} A_1'$ are in one-to-one correspondence to the SRDs on the set $\Gamma_{n-1} \setminus \{0,\overline{0}\}$ (defined analogously) starting with $1$ and having $m-1$ bar-elements which are counted by $\overline{q}_{n-2,m-1}$;
			\item  the sequences of the form $1 A_2' (a-1) a A_1'$ are in one-to-one correspondence to the SRDs on the set $\Gamma_{n-1} \setminus \{0,\overline{0}\}$ starting with $1$ and having $m$ bar-elements which are counted by $\overline{q}_{n-2,m}$.
		\end{itemize}
		The above first case can be seen from replacing $\overline{a}\overline{a-1}$ with $\overline{a-1}$ and decreasing all other elements greater than $a$ (regardless of if it has a bar) by $1$. In particular, this will lose one bar-element.
		The second case can be seen analogously, but without losing a bar-element.
		
		Conversely, for each of the $m-1$ bar-elements in the SRDs on the set $\Gamma_{n-1} \setminus \{0,\overline{0}\}$ starting with $1$, say $\overline{a-1}$ ($a>2$), we first increase all elements no less than $a$ by one, and then replace $\overline{a-1}$ with $\overline{a} \overline{a-1} $.
		Clearly, the resulting sequence is of the form $1 A_2' \overline{a} \overline{a-1} A_1'$.
		In addition, there is a unique way to transform such a sequence into an SRD of the form $0 A_1 1 A_2$, i.e.,
		$0 \overline{a-1} A_1' 1 A_2' \overline{a} $.
		So, there are $(m-1) \overline{q}_{n-2,m-1}$ SRDs lying in this situation.
		Analogously, we find there are $(n-2-m) \overline{q}_{n-2,m}$ SRDs of the form $0 {a} A_1' 1 A_2' {(a-1)} $.

		In summary, for $n\geq {3}$, the number of SRDs of type 1 and type 2 with $m$ bar-elements on $\Gamma_n$ in the form $0 A_1 1 A_2$ is given by 
		$$
		(n-1)\overline{q}_{n-1,m}+(n-m-1)\overline{q}_{n-2,m}+(m-1)\overline{q}_{n-2,m-1}.
		$$
		\textbf{\textit{case~$2$:} }$0 A_1 \overline{1} A_2$. Consider the induced sequence $1 A_1^{[r]} A_2^{[r]}$ first (Recall $A_i^{[r]}$ denotes the conjugate-reverse of $A_i$). Apparently, there are $n-m$ bar-elements in $A_1^{[r]} A_2^{[r]}$.
		
		(i) Suppose $A_1^{[r]}=\emptyset$.
		
		In this scenario, $A_2^{[r]}$ could essentially be any SRD of length $n-1$ with $n-m$ bar-elements
		the number of which is given by $\overline{q}_{n-1,n-m}+\bar{q}_{n-2,n-m}$.
		
		(ii) Suppose $A_1^{[r]}\neq\emptyset$.
		
		When $A_2^{[r]}=\emptyset$, $1 A_1^{[r]}$ is the conjugate-reverse of $A_1 \overline{1}$ thus is
		an SRD of length $n-1$. Consequently, the number of SRDs in this case is $\overline{q}_{n-1,n-m}$.
		
		Suppose $A_2^{[r]}\neq\emptyset$.
		Similar to case~$1$ (ii), there are $(n-2)\overline{q}_{n-1,n-m}$ SRDs
		where there is no $a\in[n]$ such that $A_1^{[r]}$ ends with $\overline{a}$ while $A_2^{[r]}$ starts with $\overline{a-1}$ or $A_1^{[r]}$ ends with $a-1$ while $A_2^{[r]}$ starts with $a$.
		Suppose otherwise such an $a$ exists.
		For a fixed $a\in[n]/[2]$, similar to the discussion in case~$1$ (ii), we claim that 
		\begin{itemize}
			\item  the sequences of the form $1 A_1^{[r]}{'} \overline{a} \overline{a-1} A_2^{[r]}{'}$  are in one-to-one correspondence to the SRDs on the set $\Gamma_{n-1} \setminus \{0,\overline{0}\}$ starting with $1$ and having $n-m-1$ bar-elements which are counted by $(n-m-1)\overline{q}_{n-2,n-m-1}$;
			\item  the sequences of the form $1 A_1^{[r]}{'} (a-1) a A_2^{[r]}{'}$ are in one-to-one correspondence to the SRDs on the set $\Gamma_{n-1} \setminus \{0,\overline{0}\}$ starting with $1$ and having $n-m$ bar-elements which are counted by $(m-2)\overline{q}_{n-2,n-m}$.
		\end{itemize}
		
		In summary, for $n\geq {3}$, the number of SRDs of type 1 and type 2 with $m$ bar-elements  on $\Gamma_n$ in the form $0 A_1 \overline{1} A_2$ is given by 
		$$
		n\overline{q}_{n-1,n-m}+(m-1)\overline{q}_{n-2,n-m}+(n-m-1)\overline{q}_{n-2,n-m-1}.
		$$
		Combining the above two cases together, the theorem follows.
	\end{proof}
	
	Applying Theorem~\ref{fnk}, we have
	{
		\begin{align*}
			&q_{n,m}
			=\overline{q}_{n,m}+\overline{q}_{n-1,m}\\
			=&(n-1){q}_{n-1,m}+(m-1){q}_{n-1,m-1}+(n-m-2){q}_{n-2,m}+(2m-2){q}_{n-2,m-1}\\
			&+(n-m+1)\overline{q}_{n-1,n-m}+(2n-2m-1)\overline{q}_{n-2,n-m-1}+(n-m-2)\overline{q}_{n-3,n-m-2},
		\end{align*}
	}
	and 	
		\begin{align*}
			&q_{n-1,m-1}
			=\overline{q}_{n-1,m-1}+\overline{q}_{n-2,m-1}\\
			=&(n-2){q}_{n-2,m-1}+(m-2){q}_{n-2,m-2}+(n-m-2){q}_{n-3,m-1}+(2m-4){q}_{n-3,m-2}\\
			&+(n-m+1)\overline{q}_{n-2,n-m}+(2n-2m-1)\overline{q}_{n-3,n-m-1}+(n-m-2)\overline{q}_{n-4,n-m-2}.
		\end{align*}		
	Summing up the above two equations, we can collect terms to clear all numbers of the form $\overline{q}_{x,y}$ and arrive at
	\begin{align*}
		q_{n,m}+q_{n-1,m-1}
		=&nq_{n-1,m-1}+(n-1)q_{n-1,m}+(m-2)q_{n-2,m-2}+(3n-5)q_{n-2,m-1}\\
		&+(n-m-2)q_{n-2,m}+(2m-4)q_{n-3,m-2}+(2n-2m-4)q_{n-3,m-1}.
	\end{align*}	
	Moving $q_{n-1,m-1}$ to the right-hand side, we obtain eq.~\eqref{eq:h} as desired.
	
	Is it true that there will be more signed relative derangements if we turn more unsigned elements into
	signed elements? Put it differently, is it easier to form a relative derangement if more elements have signs?
	The answer is apparently negative due to the symmetry of $q_{n,m}$.
	But, how about the cases for $m \leq n/2$?
	This is related to the unimodality of sequences.
	The sequence $x_0,x_1,x_2,\cdots,x_n$ is said to be unimodal if there exists an index $0\leq m\leq n$, called the mode of the sequence, such that $x_0\leq \cdots\leq x_{m-1}\leq x_m\geq x_{m+1}\geq\cdots\geq x_n$.
	A common and well understood approach for proving the unimodality of the sequence consisting of the coefficients of a polynomial
	is to show the roots of the polynomial are all real.
	However, this approach fails for $Q_n^B(t)$'s since some polynomials may have non-real roots.
	For instance, $Q_5^B(t)$ has only one real root $-1$, and other complex roots are approximately {$-2.5192\pm0.1281i$, $-0.3959\pm0.0201i$}.
	
	%定理15
	\begin{theorem}
		For any fixed $n\geq 1$, the sequence $q_{n,0}, q_{n,1}, \ldots, q_{n,n}$ is unimodal.
	\end{theorem}
	
	\begin{proof}
		Thanks to the symmetry of $q_{n,m}$, it suffices to prove $P(n,m)=q_{n,m}-q_{n,m-1} \geq 0$ for $m\leq n/2$,
		where we still make the convention $q_{n,m}=0$ if $m<0$.
		We shall prove this mainly by induction. 
		
		First, from the polynomials of $Q_n^B(t)$ listed in the last section, we observe that
		for $n=1,2,\ldots, 9$ and $m\leq n/2$, $P(n,m)\geq 0$.
		Secondly, we claim 
		\begin{itemize}
			\item for any $n\geq 2$,
			$P(n,1)\geq 0$;
			\item for any $n\geq 4$,
			$P(n,2) \geq 0$.
		\end{itemize}

		In order for proving $P(n,1)\geq 0$ in the case of $n\geq 2$, we construct an injection from $\mathbb{Q}_n$ to $\mathbb{Q}_{n,1}^B$ (where $\mathbb{Q}_{n,i}^B$ denotes the subset containing signed relative derangements with exactly $i$ bar-elements). For each sequence in $\mathbb{Q}_n$, replacing $n$ with $\overline{n}$, we obtain a unique sequence in $\mathbb{Q}_{n,1}^B$. Obviously, this is an injection and then $P(n,1)\geq0$ follows.
		
		Analogously, we construct an injection from $\mathbb{Q}_{n,1}^B$ to $\mathbb{Q}_{n,2}^B$ for proving $P(n,2) \geq 0$.
		We will classify the sequences in $\mathbb{Q}_{n,1}^B$ by the largest bar-element. 
		
		\textit{\textbf{case~$1$:}} If the largest bar-element in $\pi \in \mathbb{Q}_{n,1}^B$ is less than $n-1$, then we map $\pi$ to a relative derangement obtained by substituting $\overline{n}$ for $n$. In this case, the obtained relative derangements in $\mathbb{Q}_{n,2}^B$ have two bar-elements: $\overline{n}$ and $\overline {i}$ for some $1\leq i <n-1$.
		
		\textit{\textbf{case~$2$:}} If the largest bar-element in $\pi \in \mathbb{Q}_{n,1}^B$ is exactly $n-1$, and $\overline{n-1}$ is not followed by $n$, then we substitute $\overline{n}$ for $n$. In the case that $\overline{n-1}$ is followed by $n$, we replace $1$ with $\overline{1}$ to obtain a sequence in $\mathbb{Q}_{n,2}^B$. In this case, the obtained relative derangements in $\mathbb{Q}_{n,2}^B$ have two bar-elements: either $\overline{n-1}$ and $\overline{n}$, or
		$\overline{n-1}$ and $\overline{1}$ with an additional feature that $\overline{n-1}$ is followed by $n$.
		
		\textit{\textbf{case~$3$:}} Suppose the largest bar-element in $\pi \in \mathbb{Q}_{n,1}^B$ is $n$. If $n-1$ is not followed by $\overline{n}$, then we remove the bar of $n$. Meanwhile, we replace $n-1$ with $\overline{n-1}$ and $1$ with $\overline{1}$. If $\overline{n}$ follows $n-1$, then we simply replace $1$ with $\overline{1}$.
		In this case, the obtained relative derangements in $\mathbb{Q}_{n,2}^B$ have two bar-elements: either $\overline{n-1}$ and $\overline{1}$ with an additional feature that $\overline{n-1}$ is not followed by $n$, or
		$\overline{n}$ and $\overline{1}$ with the feature that $\overline{n}$ follows $n-1$.
		
		In the above mapping procedure, signed relative derangements in $\mathbb{Q}_{n,1}^B$ lying in the same case are clearly mapped to
		distinct signed relative derangements in $\mathbb{Q}_{n,2}^B$. Moreover, inspecting the patterns of the contained two bar-elements and the additional features,
		signed relative derangements from different cases are mapped to distinct signed relative derangements in $\mathbb{Q}_{n,2}^B$ (for $n\geq 4$) as well.
		Therefore, the above map is indeed an injection. Hence, $P(n,2)\geq 0$.
		
		Now suppose for $1\leq n\leq N$ and any $0 \leq m \leq n/2$,  $P(n,m) \geq 0$.
		Next, we shall show that $P(N+1,m)\geq 0$ for any $3 \leq m \leq (N+1)/2$.
		Applying Corollary~\ref{coro}, we first have
		\begin{align}
			&P(N+1,m)
			=q_{N+1,m}-q_{N+1,m-1}\nonumber\\
			=&N(q_{N,m}-q_{N,m-2})+(N-m-1)(q_{N-1,m}-q_{N-1,m-1})\nonumber\\
			&+3(N-1)(q_{N-1,m-1}-q_{N-1,m-2})+(m-3)(q_{N-1,m-2}-q_{N-1,m-3})\nonumber\\
			&+2(N-m-1)(q_{N-2,m-1}-q_{N-2,m-2})+2(m-3)(q_{N-2,m-2}-q_{N-2,m-3})\label{eq:j}.
		\end{align}
		We proceed to distinguish two cases.\\
		%\begin{itemize}
		(i) If $3\leq m\leq(N-1)/2$, we compare the two subscripts of each term $q_{x,y}$ on the RHS of eq.~\eqref{eq:j} and
		find that $y\leq x/2$. For instance, since the maximum value of $m$ here is $(N-1)/2$, as to $q_{N-1,m-2}$, we have $m-2=(N-5)/2$ which satisfies $m-2\leq(N-1)/2$. Consequently, $q_{N-1,m-2}-q_{N-1,m-3} \geq 0$ by assumption.
		Other summands are nonnegative by the same token. Therefore, $P(N+1,m)\geq0$ follows.\\
		(ii) If $N/2\leq m\leq(N+1)/2$, $m$ equals either $N/2$ or $(N+1)/2$ since $m\in\mathbb{N}$. We check the two subscripts of $q_{x,y}$ and find that $y>x/2$ in some cases. Therefore, in the following reasoning, we will make some transformation by 
		the symmetry of $q_{n,m}$.
		
		When $m=N/2$, we replace $q_{N-1,m}$ with $q_{N-1,N-m-1}$ and regroup the terms on the RHS of eq.~\eqref{eq:j}, and obtain
		\begin{align}
			P(N+1,m)\label{eq:k}
			=&N(q_{N,\frac{N}{2}}-q_{N,\frac{N-4}{2}})+\frac{N-2}{2}(q_{N-1,\frac{N}{2}}-q_{N-1,\frac{N-2}{2}})\nonumber\\
			&+3(N-1)(q_{N-1,\frac{N-2}{2}}-q_{N-1,\frac{N-4}{2}})+\frac{N-6}{2}(q_{N-1,\frac{N-4}{2}}-q_{N-1,\frac{N-6}{2}})\nonumber\\
			&+(N-2)(q_{N-2,\frac{N-2}{2}}-q_{N-2,\frac{N-4}{2}})+(N-6)(q_{N-2,\frac{N-4}{2}}-q_{N-2,\frac{N-6}{2}})\nonumber\\
			=&N(q_{N,\frac{N}{2}}-q_{N,\frac{N-4}{2}})+\frac{N-2}{2}(q_{N-1,\frac{N-2}{2}}-q_{N-1,\frac{N-2}{2}})\nonumber\\
			&+3(N-1)(q_{N-1,\frac{N-2}{2}}-q_{N-1,\frac{N-4}{2}})+\frac{N-6}{2}(q_{N-1,\frac{N-4}{2}}-q_{N-1,\frac{N-6}{2}})\nonumber\\
			&+(N-2)(q_{N-2,\frac{N-2}{2}}-q_{N-2,\frac{N-4}{2}})+(N-6)(q_{N-2,\frac{N-4}{2}}-q_{N-2,\frac{N-6}{2}})
		\end{align}
		Similarly, when $m=(N+1)/2$, we replace $q_{N,m}$ with $q_{N,N-m}$, $q_{N-1,m}$ with $q_{N-1,N-m-1}$ and $q_{N-2,m-1}$ with $q_{N-2,N-m-1}$ in eq.~\eqref{eq:j} and regroup the terms to have
		\begin{align}
			P(N+1,m)\label{eq:l}
			=&N(q_{N,\frac{N+1}{2}}-q_{N,\frac{N-3}{2}})+\frac{N-3}{2}(q_{N-1,\frac{N+1}{2}}-q_{N-1,\frac{N-1}{2}})\nonumber\\
			&+(3N-3)(q_{N-1,\frac{N-1}{2}}-q_{N-1,\frac{N-3}{2}})+\frac{N-5}{2}(q_{N-1,\frac{N-3}{2}}-q_{N-1,\frac{N-5}{2}})\nonumber\\
			&+(N-5)(q_{N-2,\frac{N-3}{2}}-q_{N-2,\frac{N-5}{2}})+(N-3)(q_{N-2,\frac{N-1}{2}}-q_{N-2,\frac{N-3}{2}})\nonumber\\
			=&N(q_{N,\frac{N-1}{2}}-q_{N,\frac{N-3}{2}})+\frac{N-3}{2}(q_{N-1,\frac{N-3}{2}}-q_{N-1,\frac{N-1}{2}})\nonumber\\
			&+(3N-3)(q_{N-1,\frac{N-1}{2}}-q_{N-1,\frac{N-3}{2}})+\frac{N-5}{2}(q_{N-1,\frac{N-3}{2}}-q_{N-1,\frac{N-5}{2}})\nonumber\\
			&+(N-5)(q_{N-2,\frac{N-3}{2}}-q_{N-2,\frac{N-5}{2}})+(N-3)(q_{N-2,\frac{N-3}{2}}-q_{N-2,\frac{N-3}{2}})\nonumber\\
			=&N(q_{N,\frac{N-1}{2}}-q_{N,\frac{N-3}{2}})+\frac{5N-3}{2}(q_{N-1,\frac{N-1}{2}}-q_{N-1,\frac{N-3}{2}})\nonumber\\
			&+\frac{N-5}{2}(q_{N-1,\frac{N-3}{2}}-q_{N-1,\frac{N-5}{2}})+(N-5)(q_{N-2,\frac{N-3}{2}}-q_{N-2,\frac{N-5}{2}}).
		\end{align}
		Inspecting term by term on the RHS of eq.~\eqref{eq:k} and eq.~\eqref{eq:l}, they are all nonnegative by assumption. Therefore, $P(N+1,m)\geq 0$. 
		%\end{itemize}
		This completes the proof of the theorem.
	\end{proof}
	
It would be interesting to provide a pure combinatorial proof for that $P(n,m) \geq 0$ for $m\leq n/2$.
Unfortunately, we are unable to achieve that at the moment.
	
	\section*{Disclosure statement}
	The authors report there are no competing interests to declare.
	
	\section*{Acknowledgements}
	
	The authors would like to thank Prof.~Yi Wang for pointing out that the roots of $Q_n^B(t)$'s are not
	necessarily all real.
	
	\appendix
	\section{Proof of Corollary~\ref{cor:dif-eq}}
	In the following, we write $\frac{\partial F}{\partial x}(x,t)$ as $F_x(x,t)$ and $\frac{\partial F}{\partial t}(x,t)$ as $F_t(x,t)$.\vspace{0.8ex} Then according to the definition of $F(x,t)$, we first have 
	$$
	F_x(x,t)=\sum\limits_{n\geq 1}nQ_n^B(t)x^{n-1},\ F_t(x,t)=\sum\limits_{n\geq 1}	{Q_n^B}'(t)x^n.
	$$
	For the terms on right-hand side of eq.~\eqref{eq:qnt-recur}, multiplying by $x^n$ and summing over $n\geq 3$, we respectively obtain	
	\begin{align*}
		\sum_{n \geq 3}(n-1)tQ_{n-1}^B(t)x^n
		&=tx^2\sum_{n \geq 3}(n-1)Q_{n-1}^B(t)x^{n-2}\\
		&=tx^2(F_x(x,t)-Q_1^B(t))\\
		\sum_{n\geq 3}(n-1)Q_{n-1}^B(t)x^n
		&=x^2\sum_{n \geq 3}(n-1)Q_{n-1}^B(t)x^{n-2}\\
		&=x^2(F_x(x,t)-Q_1^B(t))\\
		\sum_{n \geq 3}(t^3-t){Q_{n-2}^B}'(t)x^n
		&=x^2(t^3-t)\sum_{n \geq 3}{Q_{n-2}^B}'(t)x^{n-2}\\
		&=x^2(t^3-t)F_t(x,t)
	\end{align*}
	\begin{align*}
		\sum_{n\geq 3}(3n-5)tQ_{n-2}^B(t)x^n
		&=t\big{[}\sum_{n \geq 3}3nQ_{n-2}^B(t)x^n-5\sum_{n \geq 3}Q_{n-2}^B(t)x^n\big{]}\\
		&=t\big{[}\sum_{n \geq 3}3(n-2+2)Q_{n-2}^B(t)x^n-5\sum_{n \geq 3}Q_{n-2}^B(t)x^n\big{]}\\
		&=t\big{[}3\sum_{n \geq 3}(n-2)Q_{n-2}^B(t)x^n+6\sum_{n \geq 3}Q_{n-2}^B(t)x^n-5\sum_{n \geq 3}Q_{n-2}^Bx^n\big{]}\\
		&=t\big{[}3x^3\sum_{n \geq 3}(n-2)Q_{n-2}^B(t)x^{n-3}+x^2\sum_{n \geq 3}Q_{n-2}^B(t)x^{n-2}\big{]}\\
		&=t\big{[}3x^3F_x(x,t)+x^2F(x,t)\big{]}
	\end{align*}
	\begin{align*}
		\sum_{n \geq 3}(n-2)Q_{n-2}^B(t)x^n
		&=x^3\sum_{n \geq 3}(n-2)Q_{n-2}^B(t)x^{n-3}=x^3F_x(x,t)\\
		\sum_{n \geq 3}(2t^3-2t^2){Q_{n-3}^B}'(t)x^n
		&=(2t^3-2t^2)x^3\sum_{n \geq 3}{Q_{n-3}^B}'(t)x^{n-3}\\
		&=(2t^3-2t^2)x^3F_t(x,t)
	\end{align*}
	\begin{align*}
		\sum_{n \geq 3}(2n-6)tQ_{n-3}^B(t)x^n
		&=t\big{[}\sum_{n \geq 3}2nQ_{n-3}^B(t)x^n-6\sum_{n \geq 3}Q_{n-3}^B(t)x^n\big{]}\\
		&=t\big{[}2\sum_{n \geq 3}(n-3+3)Q_{n-3}^B(t)x^n-6\sum_{n \geq 3}Q_{n-3}^B(t)x^n\big{]}\\
		&=t\big{[}2\sum_{n \geq 3}(n-3)Q_{n-3}^B(t)x^n\big{]}\\
		&=t\big{[}2x^4\sum_{n \geq 3}(n-3)Q_{n-3}^B(t)x^{n-4}\big{]}\\
		&=2tx^4F_x(x,t)
	\end{align*}
	According to the computation above, for $n\geq3$, we have
	\begin{align*}
		\sum_{n \geq 3}Q_n^B(t)x^n
		=& tx^2(F_x(x,t)-Q_1^B(t))+x^2(F_x(x,t)-Q_1^B(t))+x^2(t^3-t)F_t(x,t)\\
		&+t\big[3x^3F_x(x,t)+x^2F(x,t)\big]+x^3F_x(x,t)\\
		&+(2t^3-2t^2)x^3F_t(x,t)+2tx^4F_x(x,t)\\
		=&\big{[}(t+1)x^2+(3t+1)x^3+2tx^4\big{]}F_x(x,t)\\
		&+\big{[}(t^3-t)x^2+(2t^3-2t^2)x^3\big{]}F_t(x,t)
		+tx^2F(x,t)-(t+1)^2x^2.
	\end{align*}
	Then, $F(x,t)$ is given as follows:
	\begin{align*}
		F(x,t)=&Q_1^B(t)x+Q_2^B(t)x^2+	\sum\limits_{n\geq 3}Q_n^B(t)x^n\\
		=&x+tx+t^2x^2+4tx^2+x^2+\big{[}(t+1)x^2+(3t+1)x^3+2tx^4\big{]}F_x(x,t)\\
		&+\big{[}(t^3-t)x^2+(2t^3-2t^2)x^3\big{]}F_t(x,t)+tx^2F(x,t)-(t+1)^2x^2\\
		=&\big{[}(t+1)x^2+(3t+1)x^3+2tx^4\big{]}F_x(x,t)\\
		&+\big{[}(t^3-t)x^2+(2t^3-2t^2)x^3\big{]}F_t(x,t)
		+tx^2F(x,t)+(t+1)x+2tx^2.
	\end{align*}
	After sorting out the above equations, we eventually obtain
	\begin{align*}
		&F_t(x,t)+\frac{t+1+(3t+1)x+2tx^2}{t(t^2-1)+2t^2(t-1)x}F_x(x,t)+\frac{tx^2-1}{t(t^2-1)x^2+2t^2(t-1)x^3}F(x,t)\\
		&=\frac{-1-t-2tx}{t(t^2-1)x+2t^2(t-1)x^2},
	\end{align*}
	completing the proof of Corollary~\ref{cor:dif-eq}.
	
	%参考文献

\end{document}